\documentclass[10pt]{article}

\usepackage{bigints}

\usepackage{amssymb}
\usepackage{mathrsfs}
\usepackage{amsmath}
\usepackage{amsfonts}

\usepackage{amsthm}
\newtheorem{theo}{Theorem}[section]

\newtheorem{lemm}{Lemma}[section]

\newtheorem{conj}{Conjecture}[section]
\newtheorem*{main}{Main Theorem}

\theoremstyle{definition}
\newtheorem{defi}{Definition}[section]

\theoremstyle{remark}

\newtheorem{rema}{Remark}[section]

\usepackage{hyperref}

\begin{document}

\title{\large{\textbf{Type of finite time singularities of the Ricci flow with bounded scalar curvature\footnotemark[1]}}}

\author{Shota Hamanaka\thanks{supported in doctoral program in Chuo University, Japan.}}

\date{}

\maketitle

\begin{abstract}
In this paper, we study the Ricci flow on a closed manifold of dimension $n \ge 4$ and finite time interval $[0,T)~(T < \infty)$
on which the scalar curvature are uniformly bounded.
We prove that if such flow of dimension $4 \le n \le 7$ has finite time singularities,
then every blow-up sequence of a locally Type I singularity has certain property.
Here, locally Type I singularity is what Buzano and Di-Matteo defined.
\end{abstract}

\section{Introduction}
\footnotetext[1]{The first version of this paper contained an elementary mistake. And we modified our Main Theorem to weaker one in this version.}
~~In this paper, we will consider extension problem for the Ricci flow on a closed (compact without boundary) manifold.
Given a $n$-manifold $M~(n \ge 2),$ a family of smooth Riemannian metrics $g(t)$ on $M$ is called a Ricci flow
(introduced by Hamilton in \cite{hamilton1982three}) 
on the time interval $[0,T) \subset \mathbb{R}$ 
if it satisfies 
\[
\frac{\partial g(t)}{\partial t} = - 2\,\mathrm{Ric}_{g(t)},~~\mathrm{for~all}~t \in [0,T),
\]
where $\mathrm{Ric}_{g(t)} = (R_{ij})$ denotes the Ricci curvature tensor of $g(t).$
In \cite{hamilton2formation}, Hamilton proved that
a Ricci flow on a closed manifold develops a singularity at a finite time $T$
(i.e., $T$ is the maximal existence time of the flow)
if and only if
the maximum of the norm $|\mathrm{Rm}| = \left( R_{ijkl} \cdot R^{ijkl} \right)^{1/2}$ of the Riemannian curvature tensor $\mathrm{Rm} = \mathrm{Rm}_{g(t)} = (R_{ijkl})$ blows up at $T.$
{\v{S}}e{\v{s}}um \cite{vsevsum2005curvature} proved that indeed a bound on the Ricci curvature rather than the full Riemannian curvature tensor sufficies to extend the closed Ricci flow.
Some different conditions with which the Ricci flow can be extended over $T$ are known (see \cite{matteo2020mixed},  \cite{enders2011type}, \cite{he2014remarks}, \cite{wang2008conditions}, \cite{wang2012conditions}).
It has been conjectured that a bound on the scalar curvature could potentially also be sufficient to extend the flow.
In dimension three, this is a consequence of the Hamilton-Ivey pinching estimate(see \cite{topping2006lectures})
while in higher dimensions it is known to be true for (globally) Type I Ricci flows by Ender, M{\"u}ller and Topping \cite{enders2011type} as well as in the K{\"a}hler case by Zhang \cite{zhang2010scalar}. 
In recent years, this conjecture has been the focus of many many interesting new developments, see for example 
\cite{bamler2018convergence}, \cite{bamler2017heat}, \cite{bamler2019heat}, \cite{buzano2020local}, \cite{chen2013conditions}, \cite{chen2012space}, \cite{chen2017space}, \cite{enders2011type}, \cite{simon2015extending}, \cite{simon2015some}, \cite{wang2012conditions}, \cite{zhang2010scalar}.

Motivated by Hamilton's theory \cite{hamilton2formation},
Buzano and Di-Matteo introduced a local analysis of singularities and curvature blow-up rates \cite{buzano2020local}.
Moreover, in the same paper, they also proved that the closed Ricci flow with uniformly bounded scalar curvature
has no well-behaved blow-up sequences in dimension $< 8.$

Our main result is the following.
\begin{main}
\label{maintheo}
Let $(M^{n}, g(t))_{t \in [0,T)}$ be a Ricci flow on a closed manifold $M$ of dimension $n < 8$
and $T < +\infty.$
Assume that 
\[
\sup_{M \times [0,T)} |R_{g(t)}| \le R_{0} < +\infty
\]
for some $R_{0} \in \mathbb{R}_{\ge 0},$
where $R_{g(t)}$ denotes the scalar curvature of $g(t).$
And assume that the time $T$ at which $(M, g(t))_{t \in [0,T)}$ forms singularities
and $p$ is a locally Type I singularity in the sense of Definition \ref{defi1}(3) below.
Then for any $\epsilon_{1} \in \left( 0,\frac{1}{\sqrt{n} (n-1)} \right), \epsilon_{2} \in (0,1)$ there is no blow-up sequence $(p_{i}, t_{i}) \in M \times [0,T)$
which tends to $p$ and satisfies the following :
there exists $i_{0} \in \mathbb{N}$ such that for all $i \ge i_{0}$
there exists a point 
\[
\begin{split}
(q,t) &\in B \left( p_{i}, t_{i}, \frac{\epsilon_{2}}{\sqrt{2 C}} r_{\mathrm{Ric}}(p_{i}, t_{i}) \right) \\
&\times \left( \max \left\{ 0, t_{i} - \frac{\epsilon_{2} (\sqrt{2 C})^{-1}}{1 + C_{1}} r^{2}_{\mathrm{Ric}}(p_{i}, t_{i}) \right\},
\min \left\{ T,  t_{i} + \frac{\epsilon_{2} (\sqrt{2 C})^{-1}}{1 + C_{1}} r^{2}_{\mathrm{Ric}}(p_{i}, t_{i}) \right\} \right)
\end{split}
\]
such that
\[
|\mathrm{Ric}| (q, t_{i}) \ge \epsilon_{1} r_{\mathrm{Ric}}^{-2}(p_{i},t_{i}),
\]
where $C_{1} = 2 ^{4}\sqrt{2/3}\sqrt{n-1}$ and $C$ denotes the upper bound of locally Riemann-Type I point $p$
as in Lemma \ref{sing2} (2).
\end{main}

This paper is organized as follows.
In Section 2, we summarize some fundamental results on singularity models of Ricci flows and
Ricci flow with uniformly bounded scalar curvature,
which give a background of Main Theorem. 
In Section 3, we define basic definitions and prepare some lemmas to prove our Main Theorem.
In the final Section, we prove Main Theorem.

\section{Backgrounds}
For a Ricci flow $(M, g(t)),$ we will denote by 
\[
K_{\mathrm{max}}(t) := \sup_{x \in M} |\mathrm{Rm} (x,t)|_{g(t)}.
\]
According Hamilton\cite{hamilton2formation}, one can classify finite time maximal solutions into two types below :
Let $(M, g(t))_{t \in [0,T)}~(0 < T < +\infty)$ be a finite time maximal defined solution of the Ricci flow equation.
Then, $(M, g(t))_{t \in [0,T)}$ is clearly of one and only one of the following two types.

\textbf{Type I} :~~$\sup_{t \in [0,T)} (T-t) K_{\mathrm{max} (t)} < +\infty,$

\textbf{Type II} :~~$\sup_{t \in [0,T)} (T-t) K_{\mathrm{max} (t)} = +\infty.$

For each type of maximal solution, Hamilton defines a corresponding type of limiting singularity model.
By Perelman's local non-collapsing theorem \cite{perelman2002entropy}, every Ricci flow which meet a finite time singularity has a blow-up limit which is a singularity model of the corresponding type.
A conjecture, normally attributed to Hamilton \cite{hamilton2formation}, is that a suitable blow-up sequence for a Type I singularity convergences to a nontrivial gradient shrinking soliton.
In the case where the blow-up limit is compact, this conjecture was confirmed by {\v{S}}e{\v{s}}um \cite{sesum2006convergence}.
In the general case, blow-up to a gradient shrinking soliton was proved by Naber \cite{naber2010noncompact}.
Finally, Enders-M{\"u}ller-Topping {\cite[Theorem~1.1]{enders2011type}} shown that the blow-up limits are nontrivial.
More precisely, they proved the following.
\begin{theo}[{\cite[Theorem~1.1]{enders2011type}}]
\label{EMT}
Let $(M^{n}, g(t))_{t \in [0,T)}~(T < +\infty)$ be a Type I Ricci flow and suppose $p$ is a Type I singular point(see the definition below). Then for \textrm{every} sequence $\lambda_{j} \rightarrow \infty,$
the rescaled Ricci flows $(M, g_{j}(t), p)$ defined on $[-\lambda_{j} T, 0)$ by $g_{j}(t) := \lambda_{j} g(T + \frac{t}{\lambda_{j}})$ subconverge to a normalized nontrivial gradient shrinking soliton in canonical form. 
\end{theo}

Here, Type I singular points are defined as follows.
\begin{defi}[{\cite[Definition~1.2]{enders2011type}}]
\label{EMT-type1}
A space-time sequence $(p_{i}, t_{i})$ with $p_{i} \in M$ and $t_{i} \nearrow T$ in a Ricci flow is called an \textit{essential blow-up sequence} if there exists a constant $c > 0$ such that
\[
|\mathrm{Rm}_{g(t_{i})}|_{g(t_{i})} (p_{i}) \ge \frac{c}{T - t_{i}}.
\]
A point $p \in M$ in a Type I Ricci flow is called a (general) \textit{Type I singular point} if
there exists an essential blow-up sequence with $p_{i} \rightarrow p$ on $M.$
We denote the set of all Type I singular points by $\Sigma_{G, I}.$
\end{defi}

By the work of many researchers including Hamilton and Perelman, the classification of singularity models of 3-dimensional Ricci flows are done.
On the other hand, genegal flows were relatively poorly understood.
As seen in the result of Enders-M{\"u}ller-Topping mentioned above, one important class of reasonable singularity models in higher dimensions are gradient shrinking solitons.
In dimension $n \le 3,$ all non-trivial gradient shrinking solitons are quotients of round spheres or cylinders. 
However, more complicated gradient shrinking soliton exist in dimension $\ge 4.$
In \cite{appleton2019eguchi}, Appleton constructs a class of 4-dimensional non-compact Ricci flows that develop a (globally) Type II singularity in finite time and obtains the following classification of all non-trivial blow-up singularity models 
(i.e., with blow-up factors $\lambda_{i}$ such that $\lambda_{i} \rightarrow \infty$ are arbitrary):

\noindent
(1)~The Eguchi-Hanson metric which is Ricci flat and asymptotic to the flat cone  $\mathbb{R}^{4}/\mathbb{Z}_{2}.$

\noindent
(2)~The flat cone $\mathbb{R}^{4}/\mathbb{Z}_{2}$ which has an isolated orbifold singularity at the origin.

\noindent
(3)~The quotient $M_{Bry}/\mathbb{Z}_{2}$ of the Bryant soliton which also has an isolated orbifold singularity at its tip.

\noindent
(4)~The cylinder $\mathbb{R}P^{3} \times \mathbb{R}.$
 
Appleton's examples also show that a Ricci flow on a 4-dimensional non-compact manifold may collapse an embedded two-dimensional sphere with self-intersection $k \in \mathbb{Z}$ to a point in finite time and thereby produce a singularity.
In addition, there are further examples in higher dimensions \cite{stolarski2019curvature} whose only blow-up models that are gradient shrinking solitons must be singular and possibly degenerate (since a Ricci flat cone may be considered as a gradient shrinking Ricci soliton (see {\cite[Theorem~1.1,~p.4]{stolarski2019curvature}} and {\cite[Proposition~4.3.1]{petersen2006riemannian}})).
Moreover, the scalar curvatures of Stolarski's examples in \cite{stolarski2019curvature} satisfy a Type I bound, i.e.,
\[
\limsup_{t \nearrow T}~(T-t)  \sup_{x \in S^{p} \times S^{q+1}} |R|(x,t) < +\infty,
\]
where $p \ge 2, q \ge 10.$
This consistent with the convergence results for Ricci flows with uniform scalar curvature bounds proved in {\cite[Theorem~1.2]{bamler2018convergence}} as follows (c.f. the result of Enders-M{\"u}ller-Topping (Theorem \ref{EMT} mentioned above)).
\begin{theo}[{\cite[Theorem~1.2]{bamler2018convergence}}]
Let $(M^{n}, g(t))_{t \in [0,T)}~(T < +\infty,~n \ge 4)$ be a closed Ricci flow on the maximal finite time interval $[0,T).$
Assume that for some constant $C < +\infty,$ we have 
\[
\sup_{x \in M}~R(x, t) \le C (T-t)^{-1}~~~\mathrm{for~all}~~t \in [0,T).
\]
Then for any point $q \in M$ and any sequence of times $t_{i} \nearrow T,$
we can choose a subsequence such that $(M, (T-t_{i})^{-1} g_{t_{i}}, q)$ converges to a pointed, singular space
(see {{\cite[Definition~2.1]{bamler2018convergence}}} for the definition) $(\mathcal{X}, q_{\infty}) = (X, d, \mathcal{R}, g, q_{\infty})$ that has singularities of codimension 4 in the sense of {\cite[Definition~2.2]{bamler2018convergence}}
and that is $Y$-regular at scale 1 in the sense of {\cite[Definition~2.4]{bamler2018convergence}} for some $Y < +\infty$
that only depends on $g(0)$ and $C.$
Moreover, $\mathcal{X}$ is a gradient shrinking soliton in the following sense :
There is a smooth and bounded function $f_{\infty} \in C^{\infty}(\mathcal{R})$ that satisfies the shrinking soliton equation
\[
\mathrm{Ric}_{g} + \nabla^{2} f_{\infty} = \frac{1}{2} g~~\mathrm{on}~~\mathcal{R}.
\]
\end{theo}
Motivated by these results, Bamler conjectured the following.
\begin{conj}[{\cite[Conjecture~5.2]{bamler2021recent}}]
For any Ricci flow ``most'' singularity models are gradient shrinking solitons that may be degenerate and may have a singular set of codimension $\ge 4.$
\end{conj}
In dimension three, the Hamilton-Ivey pinching estimate says that 
the lowest sectional curvature is bounded by the scalar curvature
and is also important to analyze finite-time singularities.
On the other hand, in higher dimensions, there is no longer a corresponding estimate.
But the behavior of the scalar curvature under the Ricci flow is still important.
At the end of this section, we will describe the position of our main result of this paper in comparison with the result in \cite{enders2011type}.
In \cite{enders2011type}, Enders-M{\"u}ller-Topping defined some notions of singularity.
\begin{defi}[{\cite[Definition~1.3,~1.4]{enders2011type}}]
(1)~A point $p \in M$ in a Type I Ricci flow is called a \textit{special Type I singular point} if there exists an essential blow-up sequence $(p_{i}, t_{i})$ with $p_{i} = p$ for all $i \in \mathbb{N}.$
The set of all such points is denoted by $\Sigma_{s}.$
Moreover, we denote by $\Sigma_{\mathrm{Rm}} \subset \Sigma_{s}$ the set of points $p \in M$ for which $|\mathrm{Rm}_{g(t)}|_{g(t)}(p)$ blows up at the Type I rate as $t \rightarrow T.$

\noindent
(2)~The set $\Sigma_{R}$ is defined to be the set of points $p \in M$ for which the scalar curvature $R_{g(t)}(p)$ blows up at the Type I rate as $t \rightarrow T.$
\end{defi}

From the definitions, it is clear that a priori
\[
\Sigma_{R} \subset \Sigma_{\mathrm{Rm}} \subset \Sigma_{s} \subset \Sigma_{G, I} \subset \Sigma,
\]
where $\Sigma_{G, I}, \Sigma$ denote respectively the set in Definition \ref{EMT-type1} and Definition \ref{defi1}(1) below.
Thus, Enders-M{\"u}ller-Topping proved the following.
\begin{theo}[{\cite[Theorem~1.2]{enders2011type}}]
Let $(M^{n}, g(t))_{t \in [0,T)}~(T < +\infty)$ be a Type I Ricci flow.
Then $\Sigma = \Sigma_{R}.$
Consequently, we have 
\[
\Sigma_{R} = \Sigma_{\mathrm{Rm}} = \Sigma_{s} = \Sigma_{G, I} = \Sigma.
\]
\end{theo}
Hence, from Hamliton's compactness theorem, if $(M^{n}, g(t))_{t \in [0,T)}~(T < +\infty,~n \ge 4)$ is a Type I Ricci flow with uniformly bounded scalar curvature (in space and time), then the flow can be extended over $T.$
In other words, $\mathbf{(*)}$ \textit{if} $(M^{n}, g(t))_{t \in [0,T)}~(T < +\infty,~n \ge 4)$ \textit{with uniformly bounded scalar curvature has a finite time singularity at time} $T,$ \textit{then such flow must be Type II}.

From the perspective of globally Type I or Type II, we can only analyze singularities of such Types.
To analyze pointwisely singularities of Ricci flows,
Buzano and Di-Matteo \cite{buzano2020local} introduced general definitions of locally Type I and locally Type II singular points (see Definition \ref{defi1} below).
Our main result in this paper is implicitly directed towards proving the same assertion in $\mathbf{(*)}$ in the local sense in dimension $< 8$ (exact claim is Main Theorem mentioned above).

\section{Definitions and some known facts}
\begin{defi}[{\cite[Definition~1.1]{buzano2020local}}]
\label{defi1}
Let $(M^{n}, g(t))_{t \in [0,T)}$ be a Ricci flow maximally defined on the interval $[0,T),~T < +\infty$
and assume that $(M, g(t))$ has bounded (resp. Ricci) curvatures for $t \in [0,T).$
For any $t \in [0,T),$ we consider the rescaled Ricci flow $\tilde{g}_{t}(s) := (T-t)^{-1} g (t + (T-t)s)$
defined for $s \in [-\frac{1}{T-t}, 1).$

\noindent
(1)~We say that a point $p \in M$ is a \textit{singular~point} (resp. \textit{Ricci singular point})
if for any neighbourhood $U$ of $p,$ the Riemannian curvature (resp. Ricci curvature) becomes unbounded on $U$ as $t$ approaches $T.$
The \textit{singular~set} $\Sigma$ (resp. \textit{Ricci singular set} $\Sigma^{\mathrm{Ric}}$) is the set of all such points and the \textit{regular~set} $\mathcal{R}$ (resp. \textit{Ricci regular set} $\mathcal{R}^{\mathrm{Ric}}$) consists of the complement of $\Sigma$ (resp. $\Sigma^{\mathrm{Ric}}$) in $M,$ i.e., $\mathcal{R} := M \setminus \Sigma$
(resp. $\mathcal{R}^{\mathrm{Ric}} := M \setminus \Sigma^{\mathrm{Ric}}$).

\noindent
(2)~We say that a point $p \in M$ is a \textit{locally Type~I~singular~point} (resp. \textit{locally Type I Ricci singular point}) if there exist constants 
$c_{I},~C_{I},~r_{I} > 0$ such that we have
\[
0 < c_{I} < \limsup_{t \nearrow T} \sup_{B_{\tilde{g}_{t}(0)}(p,r_{I}) \times (-r_{I}^{2}, r_{I}^{2})}  |\mathrm{Rm}_{\tilde{g}_{t}}|_{\tilde{g}_{t}} \le C_{I} < +\infty.
\]  
\[
\left( \mathrm{resp.}~~0 < a_{0} c_{I} < \limsup_{t \nearrow T} \sup_{B_{\tilde{g}_{t}(0)}(p,r_{I}) \times (-r_{I}^{2}, r_{I}^{2})}  |\mathrm{Ric}_{\tilde{g}_{t}}|_{\tilde{g}_{t}} \le a_{0} C_{I} < +\infty. \right)
\]
Here $a_{0} := \sqrt{n}(n-1).$
We denote the set of such points by $\Sigma_{I}$ (resp. $\Sigma_{I}^{\mathrm{Ric}}$) and call it the \textit{locally~Type~I~singular~set} (resp. \textit{locally~Type I Ricci singular set}).

\noindent
We say that a point $p \in M$ is a \textit{locally~Type~II~singular~point} (resp. \textit{locally Type II Ricci singular point}) if for any $r > 0$ we have
\[
\limsup_{t \nearrow T} \sup_{B_{\tilde{g}_{t}(0)}(p,r_{I}) \times (-r_{I}^{2}, r_{I}^{2})}  |\mathrm{Rm}_{\tilde{g}_{t}}|_{\tilde{g}_{t}} = \infty.
\]
\[
\left( \mathrm{resp.}~~\limsup_{t \nearrow T} \sup_{B_{\tilde{g}_{t}(0)}(p,r_{I}) \times (-r_{I}^{2}, r_{I}^{2})}  |\mathrm{Ric}_{\tilde{g}_{t}}|_{\tilde{g}_{t}} = \infty. \right)
\]
We denote the set of such points by $\Sigma_{II}$ (resp. $\Sigma_{II}^{\mathrm{Ric}}$) and call it the \textit{locally Type~II~singular~set} (resp. \textit{locally Type II Ricci singular set}).
\end{defi}

\begin{rema}
(a)~A Ricci flow on $M \times [0,T)$ is said to be of (globally) \textit{Type I} if there exists a constant $C$ such that 
\[
(T - t) |\mathrm{Rm}(\cdot, t)|_{g(t)} \le C,~~~\forall t \in [0,T).
\]
Note that if the Ricci flow is Type I in this global sense, then so are all the rescaled flows $(M, \tilde{g}_{t})$ 
with the same constant $C.$
Hence we obtain the upper bound in Definition \ref{defi1} (2) for any radius $r_{I} < 1$ with $C_{I} = \frac{C}{1 - r^{2}_{I}}.$
On the other hand, by {\cite[Theorem~3.2]{enders2011type}},
we also obtain the lower bound in it.
Therefore, all singular points in such flow are locally Type I in the sense of Definition \ref{defi1}(2).

\noindent
(b)~Very recently, Hallgren \cite{hallgren2021riccivolume} showed that locally Type II singularity in the sense of Definition \ref{defi1}(3) below
may appear as the singularities of a Ricci flow with bounded Ricci curvature from below and non-collapsed volume(see {\cite[Theorem~1.7]{hallgren2021riccivolume}}).
\end{rema}

\begin{defi}[{\cite[Definition~1.3~Definition~4.1]{buzano2020local}}]
\label{defi2}
Let $(M^{n}, g(t))_{t \in [0,T)}$ be a Ricci flow defined on the interval $[0,T)$
and let $(p,t) \in M \times [0,T)$ be a space-time point.

\noindent
(1)~For $r > 0,$ we define the \textit{parabolic~cylinder} $P(p,t,r)$ with centre $(p,t)$ and radius $r$ by
\[
P(p,t,r) := B_{g(t)} (p,r) \times (\max \{ t - r^{2},~0 \},~\min \{ t + r^{2},~T \}).
\]

\noindent
(2)~We define the \textit{Riemann~scale} $r_{\mathrm{Rm}}(p,t)$ at $(p,t)$ by
\[
r_{\mathrm{Rm}} (p,t) := \sup \left\{ r > 0 | |\mathrm{Rm}| < r^{-2}~\mathrm{on}~P(p,t,r) \right\}.
\]
If $(M, g(t))$ is flat for every $t \in [0,T),$ we set $r_{\mathrm{Rm}} := +\infty.$
Moreover, by slight abuse of notation, we may sometimes write $P(p,t,r_{\mathrm{Rm}})$ for $P(p,t,r_{\mathrm{Rm}}(p,t)).$

\noindent
(3)~We define the \textit{time-slice~Riemann~scale} $\tilde{r}_{\mathrm{Rm}}(p,t)$ at $(p,t)$ by
\[
\tilde{r}_{\mathrm{Rm}} (p,t) := \sup \left\{ r > 0 | |\mathrm{Rm}| < r^{-2}~\mathrm{on}~B_{g(t)}(p,r) \right\}.
\]
When the flow is flat at time $t,$ we set $\tilde{r}_{\mathrm{Rm}}(p,t) := +\infty.$
Clearly $\tilde{r}_{\mathrm{Rm}}(p,t) \ge r_{\mathrm{Rm}}(p,t).$

\noindent
(4)~We define the \textit{Ricci~scale} $r_{\mathrm{Ric}}(p,t)$ at $(p,t)$ by 
\[
r_{\mathrm{Ric}}(p,t) := \sup \left\{ r > 0 | |\mathrm{Ric}| < a_{0}(n) r^{-2}~\mathrm{on}~P(p,t,r) \right\},
\]
where $a_{0} := \sqrt{n}(n-1)$ as before, or equivalently by
\[
r_{\mathrm{Ric}}(p,t) := \sup \left\{ r > 0 | -(n-1)r^{-2}g < \mathrm{Ric} < (n-1) r^{-2} g~\mathrm{on}~P(p,t,r) \right\}.
\]
If $(M,g(t))$ is Ricci flat for every $t \in [0,T),$ we set $r_{\mathrm{Ric}}(p,t) := +\infty.$
Moreover, by slight abuse of notation, we may sometimes write $P(p,t, r_{\mathrm{Ric}})$ for $P(p,t,r_{\mathrm{Ric}}(p,t)).$

\noindent
(5)~The \textit{time-slice~Ricci~scale} at $(p,t)$ is given by $\tilde{r}_{\mathrm{Ric}}(p,t) := +\infty$ if the flow is Ricci flat, otherwise we set
\[
\tilde{r}_{\mathrm{Ric}} (p,t) := \left\{ r > 0 | |\mathrm{Ric}| < a_{0} r^{-2}~\mathrm{on}~B_{g(t)}(p,r) \right\},
\]
where $a_{0}$ denotes the same constant in (4).    
\end{defi}

\begin{defi}[{\cite[(1.8)]{buzano2020local}}]
\label{defi3}
Let $(M^{n}, g(t))_{t \in [0,T)}$ be a Ricci flow maximally defined on the interval $[0,T),~T < +\infty$
and assume that $|\mathrm{Ric}|$ is not identically zero.
For every $\delta \in (0,1),$ we define the set of $\delta$-\textit{well-behaved~points}
\[
G_{\delta} = G_{\delta, t_{1}} := \left\{ q \in M | \delta a_{0} r^{-2}_{\mathrm{Ric}}(q,t) < |\mathrm{Ric}|(q,t),~\mathrm{for~all}~t \in[t_{1},T) \right\},
\]
where $a_{0} := \sqrt{n}(n-1).$
A sequence of space-time points $(p_{i}, t_{i}) \in M \times [0,T)$ with $t_{i} \nearrow T$ and $r_{\mathrm{Ric}}(p_{i}, t_{i}) \rightarrow 0$ is $\delta$-well-behaved if
for sufficiently large $i$ and for any $q \in B(p_{i}, t_{i}, \sqrt{\delta} r_{\mathrm{Ric}} (p_{i}, t_{i}))$ we have
\[
\delta a_{0} r_{\mathrm{Ric}}^{-2} (q, t_{i}) < |\mathrm{Ric}|(q, t_{i}).
\]
\end{defi}
\begin{rema}
\label{rema-well}
The definition of the $\delta$-well-behaved sequence is slightly different to one in \cite{buzano2020local}.
In \cite{buzano2020local}, it requires that the last estimate holds on $B(p_{i}, t_{i}, \sqrt{\delta} \tilde{r}_{\mathrm{Ric}} (p_{i}, t_{i})).$
But it is sufficient that the following Theorem holds.
\end{rema}
The constant $a_{0}$ in Definition \ref{defi3} is added for convenience, since it implies $r_{\mathrm{Ric}} \ge r_{\mathrm{Rm}}.$

\begin{theo}[{\cite[Theorem~1.11]{buzano2020local}}]
\label{buzano-matteo}
Let $(M^{n}, g(t))_{t \in [0,T)}$ be a Ricci flow on a closed manifold $M$ of dimension $n < 8$
and $T < +\infty.$
Assume that 
\[
\sup_{M \times [0,T)} |R_{g(t)}| \le n(n-1)R_{0} < +\infty
\]
for some $R_{0} \in \mathbb{R}_{\ge 0},$
where $R_{g(t)}$ denotes the scalar curvature of $g(t).$
Then for any $t_{1} \in (0,T)$ and any $\delta \in (0,1)$
there cannot be any $\delta$-well-behaved blow-up sequences.
Moreover, if $M = G_{\delta}$ for some $\delta > 0,$ then the flow can be smoothly extended over time $T.$
\end{theo}

\begin{lemm}[{\cite[Theorem~1.2,~Theorem~1.9,~Corollary~1.10,~Corollary~4.10]{buzano2020local}}]
\label{sing1}
Let $(M^{n}, g(t))_{t \in [0,T)}$ be a Ricci flow on a manifold $M$ of dimension $n,$
maximally defined on $[0,T),$ $T < +\infty,$
and with complete bounded curvature time-slices.
Assume that $(M, g(0))$ satisfies $inj (M, g(0)) > 0,$
where $inj(M, g(0))$ denotes the injectivity radius of $(M, g(0)).$
Then
\[
\Sigma = \Sigma_{I} \cup \Sigma_{II} = \Sigma^{\mathrm{Ric}} = \Sigma^{\mathrm{Ric}}_{I} \cup \Sigma^{\mathrm{Ric}}_{II},
\]
\[
\Sigma_{I} \subset \Sigma^{\mathrm{Ric}}_{I},~~\Sigma^{\mathrm{Ric}}_{II} \subset \Sigma_{II}.
\]
Therefore
\[
\Sigma_{II} \setminus \Sigma^{\mathrm{Ric}}_{II} = \Sigma^{\mathrm{Ric}}_{I} \setminus \Sigma_{I}.
\]
\end{lemm}
In paticular, 
\[
(*)~~~\mathrm{if}~p \in \Sigma_{I},~\mathrm{then}~p \in \Sigma^{\mathrm{Ric}}_{I}.
\]

\begin{lemm}[{\cite[Theorem~1.4,~Theorem~4.7]{buzano2020local}}]
\label{sing2}
Let $(M^{n}, g(t))_{t \in [0,T)}$ be a Ricci flow on a manifold $M$ of dimension $n,$
maximally defined on $[0,T),$ $T < +\infty,$
and with bounded curvature time-slices.
Then the following holds :

\noindent
(1)~$p \in \Sigma$ (resp. $\Sigma^{\mathrm{Ric}}$) if and only if $\limsup_{t \nearrow T} r^{-2}_{\mathrm{Rm}} (p,t)$ 
(resp. $r^{-2}_{\mathrm{Ric}} (p,t)) = +\infty,$

\noindent
(2)~$p \in \Sigma_{I}$ (resp. $\Sigma^{\mathrm{Ric}}_{I}$) if and only if 
\[
\tilde{c}_{I} < \limsup_{t \nearrow T} (T - t) r^{-2}_{\mathrm{Rm}} (p,t) < \tilde{C}_{I}
\]
\[
\left( \mathrm{resp.}~~\tilde{c}_{I} < \limsup_{t \nearrow T} (T - t) r^{-2}_{\mathrm{Ric}} (p,t) < \tilde{C}_{I} \right)
\]
for some $0 < \tilde{c}_{I}, \tilde{C}_{I},$

\noindent
(3)~$p \in \Sigma_{II}$ (resp. $\Sigma^{\mathrm{Ric}}_{II}$) if and only if 
\[
\limsup_{t \nearrow T} r^{-2}_{\mathrm{Rm}} (p,t)  = +\infty.
\]
\[
\left( \mathrm{resp.} \limsup_{t \nearrow T} r^{-2}_{\mathrm{Ric}} (p,t)  = +\infty. \right)
\]
\end{lemm}

We also have a non-oscillation of locally Type I singularities :
\begin{lemm}[{\cite[Corollary~2.4]{buzano2020local}}]
\label{sing3}
Let $(M^{n}, g(t))_{t \in [0,T)}$ be a Ricci flow as in the previous lemma.
Assume that $p \in \Sigma.$
Then 
\[
r^{-2}_{\mathrm{Rm}} (p,t) > \frac{1}{T - t}~~\forall t \in [0,T).
\]
In paticular, $\liminf_{t \nearrow T} (T-t) r^{-2}_{\mathrm{Rm}}(p,t) \ge 1.$
\end{lemm}

From $(*),$ Lemma \ref{sing2} and Lemma \ref{sing3}, we have that if $p \in \Sigma_{I} (\subset \Sigma^{\mathrm{Ric}}_{I}),$
\[
\begin{split}
\liminf_{t \nearrow T} \frac{(T-t)r^{-2}_{\mathrm{Ric}}(p,t)}{(T-t)r^{-2}_{\mathrm{Rm}}(p,t)} 
&\ge \frac{\liminf_{t \nearrow T}(T-t)r^{-2}_{\mathrm{Ric}}(p,t)}{\limsup_{t \nearrow T} (T-t)r^{-2}_{\mathrm{Rm}}(p,t)} \\
&\ge \frac{1}{\limsup_{t \nearrow T}(T-t)r^{-2}_{\mathrm{Rm}}(p,t)} \\
&\ge \tilde{C}^{-1}_{I} > 0.
\end{split}
\]
Based on this, we will consider the following condition :
\[
(\mathrm{Rm/Ric})~~~\liminf_{t \nearrow T} \frac{r^{-2}_{\mathrm{Ric}}(p,t)}{r^{-2}_{\mathrm{Rm}}(p,t)} \ge C > 0
\]
for some positive constant $C > 0.$

Next, we will describe some properties of Riemann/Ricci scale.

\begin{lemm}[{\cite[Theorem~2.2,~Theorem~4.4]{buzano2020local}}]
\label{prop1}
Let $(M^{n}, g(t))_{t \in [0,T)}$ be a complete Ricci flow.
Then for any pair of space-time points $(p,t)$ and $(q,s)$ we have
\[
|r_{\mathrm{Rm}}(p,t) - r_{\mathrm{Rm}}(q,s)| \le \min \{ d_{t}(p,q),~d_{s}(p,q) \} + C_{1} |t -s|^{1/2}
\]
and
\[
|r_{\mathrm{Ric}}(p,t) - r_{\mathrm{Ric}}(q,s)| \le \min \{ d_{t}(p,q),~d_{s}(p,q) \} + C_{1} |t -s|^{1/2}
\]
where $C_{1} := 2 ^{4}\sqrt{2/3}\sqrt{n-1}.$
\end{lemm}

\begin{lemm}[{\cite[Corollary~2.3,~Corollary~4.5]{buzano2020local}}]
\label{prop2}
Let $(M^{n}, g(t))_{t \in [0,T)}$ be a complete Ricci flow defined maximally on $[0,T)$
and let $(p,t) \in M \times [0,T).$
Then we have
\[
4^{-1} r^{-2}_{\mathrm{Rm}}(p,t) \le r^{-2}_{\mathrm{Rm}}(\cdot, \cdot) \le 4 r^{-2}_{\mathrm{Rm}}(p,t)~~\mathrm{on}~P(p,t,a_{1} r_{\mathrm{Rm}}(p,t))
\]
and
\[
4^{-1} r^{-2}_{\mathrm{Ric}}(p,t) \le r^{-2}_{\mathrm{Ric}}(\cdot, \cdot) \le 4 r^{-2}_{\mathrm{Ric}}(p,t)~~\mathrm{on}~P(p,t,a_{1} r_{\mathrm{Ric}}(p,t)).
\]
Here, $a_{1}$ denotes a dimensional constant satisfying $a_{1} \le \frac{1}{2(1+C_{1})}$
where $C_{1}$ is the constant in Lemma \ref{prop1}.
\end{lemm}
\begin{rema}
\label{rema-harnack}
On the same time-slice, we can take $a_{1} = 1/2$ in the above lemma.
That is, we have
\[
4^{-1} r^{-2}_{\mathrm{Rm}}(p,t) \le r^{-2}_{\mathrm{Rm}}(\cdot, t) \le 4 r^{-2}_{\mathrm{Rm}}(p,t)~~\mathrm{on}~B(p,t, 1/2 r_{\mathrm{Rm}}(p,t))
\]
and
\[
4^{-1} r^{-2}_{\mathrm{Ric}}(p,t) \le r^{-2}_{\mathrm{Ric}}(\cdot, t) \le 4 r^{-2}_{\mathrm{Ric}}(p,t)~~\mathrm{on}~B(p,t, 1/2 r_{\mathrm{Ric}}(p,t)).
\]
\end{rema}

Next, we will describe some properties of the Ricci flow with uniformly bounded scalar curvature.

\begin{lemm}[{\cite[Theorem~1.1]{bamler2019heat}}~see~also~\cite{simon2015extending}]
\label{prop3}
For any $A \in (0, \infty)$ and $n \in \mathbb{N},$ there exists a constant $C = C(A, n) < \infty$
such that the following holds :
Let $(M^{n}, g(t))_{t \in [0,T)}$ be a $n$-dimensional closed Ricci flow with $\nu[g(0), 1 + A^{-1}] \ge - A,$
where $\nu[g(0), 1 + A^{-1}]$ denotes the $\nu$-invariant (see {\cite[Section~2.1]{bamler2018convergence}}).
Assume that $|R| \le R_{0}$ on $M \times [0,T)$ for some constant $0 \le R_{0} \le A.$
Then for any $0 \le t_{1} \le t_{2} \le T$ and $x, y \in M$ we have the distance bounds
\[
d_{t_{1}} (x,y) - C \sqrt{t_{2} - t_{1}} \le d_{t_{2}} (x,y) \le \exp \left( C R_{0}^{1/2} \sqrt{t_{2} - t_{1}} \right) d_{t_{1}}(x,y) + C \sqrt{t_{2} - t_{1}}.
\]
In paticular, if $\min \{ d_{t_{1}}(x,y),~d_{t_{2}}(x,y) \} \le D$ for some $D < +\infty,$ then
\[
|d_{t_{1}}(x,y) - d_{t_{2}}(x,y)| \le \bar{C} (A, n, D) \sqrt{t_{2} - t_{2}}.
\]
\end{lemm}

\begin{lemm}[{\cite[Theorem~1.1]{zhang2011bounds}}~and~\cite{bamler2017heat}~or~\cite{simon2015extending}]
\label{prop4}
Let $(M^{n}, g(t))_{t \in [0,T)}$ be a $n$-dimensional closed Ricci flow with $|R| \le R_{0} < +\infty$ on $M \times [0,T).$
Then there exist constants $0 < \kappa_{1} = \kappa_{1} (n, \nu[g(0), 2T])  \le \kappa_{2} = \kappa_{2} (n, \nu[g(0), 2T]) < +\infty$ such that 
\[
\kappa_{1} r^{n} \le \mathrm{Vol}_{g(t)} (B(x,t,r)) \le \kappa_{2} r^{n}
\]
for all $0 < r < \sqrt{t}, t \in (0,T).$
\end{lemm}

\begin{lemm}[\cite{bamler2017heat}~or~\cite{simon2015extending}]
\label{prop5}
Let $(M^{n}, g(t))_{t \in [0,T)}$ be a $n$-dimensional closed Ricci flow with $|R| \le R_{0} < +\infty$ on $M \times [0,T).$
Then there exists a positive constant $1 < V = V(M, g(0), R_{0}) < +\infty$ such that 
\[
0 < V^{-1} \le \mathrm{Vol} (M, g(t)) \le V < +\infty.
\]
\end{lemm}

By Lemma \ref{prop4} and \ref{prop5}, we obtain the following diameter bounds :

\begin{lemm}[\cite{bamler2017heat}~or~\cite{simon2015extending}]
\label{prop6}
Let $(M^{n}, g(t))_{t \in [0,T)}$ be a $n$-dimensional closed Ricci flow with $|R| \le R_{0} < +\infty$ on $M \times [0,T).$
Then there exists a positive constant $1 < d = d(M, g(0), R_{0}, T) < +\infty$ such that 
\[
0 < d^{-1} \le \mathrm{diam} (M, g(t)) \le d < +\infty.
\]
\end{lemm}

\noindent
By Lemma \ref{prop6}, we can take the constant $\bar{C}$ in Lemma \ref{prop3} uniformly.

Next, we state the first key fact proved by Bamler and Zhang \cite{bamler2017heat} to prove Main Theorem.

\begin{lemm}[{\cite[Lemma~6.1]{bamler2017heat}}]
\label{prop7}
For $m = 0,1, \cdots,$ there exist $A_{0}, A_{1}, \cdots < +\infty$ such that the following holds :
Let $(M^{n}, g(t))_{t \in [0,T)}$ be a closed Ricci flow and let $x_{0} \in M,$ $0 < r^{2}_{0} < \min \{ 1, t_{0} \}.$
Assume that $B(x_{0}, t_{0}, r_{0})$ is relative compact,
$|R| \le n$ and $|\mathrm{Rm}| \le r^{-2}_{0}$ on $P^{-}(x_{0}, t_{0}, r_{0}, -r^{2}_{0}),$
where $P^{-}(x_{0}, t_{0}, r_{0}, -r^{2}_{0})$ denotes the backward parabolic cylinder centered at $(x_{0}, t_{0})$ of radius $r_{0}$ defined by
\[
P^{-}(x_{0}, t_{0}, r_{0}, -r^{2}_{0}) := B(x_{0}, t_{0}, r_{0}) \times [t_{0} - r^{2}_{0}, t_{0}].
\]
Then
\[
\begin{split}
&|\nabla^{m} \mathrm{Ric}|(x_{0}, t_{0}) \le A_{m} r^{-1-m}_{0}~\mathrm{for~all}~m \ge 0, \\
&|\partial_{t} \mathrm{Rm}|(x_{0}, t_{0}) \le A_{0} r^{-3}_{0}.
\end{split}
\]
\end{lemm}

\begin{rema}
\label{rema1}
By the variation of the Ricci curvature under the Ricci flow :
\[
\tilde{\nabla}_{\partial_{t}} \mathrm{Ric} (\cdot, \bullet) = \Delta \mathrm{Ric}(\cdot, \bullet)
+ \sum_{i,j} 2 \langle R(\cdot, e_{i})e_{j}, \bullet \rangle \mathrm{Ric}(e_{i}, e_{j})
\]
we can also obtain
\[
|\partial_{t} \mathrm{Ric}| (x_{0}, t_{0}) \le A_{0} r^{-3}_{0}.
\]
Here, $(e_{i})$ denotes an orthonomal frame at a point in $M$ and
$\tilde{\nabla}$ denotes the connection on $\pi^{*} TM,$ where 
$\pi : M \times [0,T) \rightarrow M$ is the natural projection
defined as follows :
For spatial vector fields $X, Y,$ we define $\tilde{\nabla} :=$ the Levi-Civita connection of $g(t).$
For spatial vector field $X,$ we define 
\[
\tilde{\nabla}_{\partial_{t}} X := \frac{d}{dt} X - \mathrm{Ric} (X).
\]
\end{rema}

The second key lemma to prove Main Theorem is the following one.
\begin{lemm}
\label{prop8}
Let $(M^{n}, g(t))_{t \in [0,T)}$ be a $n$-dimensional closed Ricci flow with $|R| \le R_{0} < +\infty$ on $M \times [0,T).$
Then the assertion for $r_{\mathrm{Rm}}$ in Lemma \ref{sing2} $(1)$ is holds when 
we replace $\limsup$ with $\liminf$,
that is,
\[
p \in \Sigma~\mathit{if~and~only~if}~\liminf_{t \nearrow T} r^{-2}_{\mathrm{Rm}} (p,t) = \infty.
\]
\end{lemm}

\begin{proof}
From the distance distortion estimate (Lemma \ref{prop3}) and the backward pseudolocality theorem by Bamler and Zhang{\cite[Theorem~1.5]{bamler2017heat}} 
(or from Lemma \ref{prop3}, Lemma \ref{prop4} and {\cite[Theorem~1.48]{bamler2020structure}}),
we can prove that
$p \in \Sigma$ if and only if $\limsup_{t \nearrow T} \tilde{r}_{\mathrm{Rm}}(p,t) = 0$
in the same way used in the proof of {\cite[Corollary~1.11]{bamler2017heat}}.
Since $\tilde{r}_{\mathrm{Rm}} \ge r_{\mathrm{Rm}},$ we have that
$p \in \Sigma$ if and only if $\limsup_{t \nearrow T} r_{\mathrm{Rm}}(p,t) = 0.$
\end{proof}

\section{Proof of Main Theorem}
Here, we will show Main Theorem.
To do this, we shall show the following slight general claim.
\begin{theo}
\label{theo1}
Let $(M^{n}, g(t))_{t \in [0,T)}$ be a Ricci flow on a closed manifold $M$ of dimension $n < 8$
and $T < +\infty.$
Assume that 
\[
\sup_{M \times [0,T)} |R_{g(t)}| \le R_{0} < +\infty
\]
for some $R_{0} \in \mathbb{R}_{\ge 0}.$
And assume that the time $T$ is finite time singularity of $(M, g(t))_{t \in [0,T)}.$
Assume that a singular point $p \in \Sigma$ satisfies
\[
(\mathrm{Rm/Ric})~~~\liminf_{t \nearrow T} \frac{r^{-2}_{\mathrm{Ric}}(p,t)}{r^{-2}_{\mathrm{Rm}}(p,t)} \ge C > 0
\]
for some $C > 0.$
Then for any $\epsilon_{1} \in \left( 0,\frac{1}{\sqrt{n} (n-1)} \right), \epsilon_{2} \in (0,1)$ there is no blow-up sequence $(p_{i}, t_{i}) \in M \times [0,T)$
which tends to $p$ and satisfies the following :
there exists $i_{0} \in \mathbb{N}$ such that for all $i \ge i_{0}$
there exists a point 
\[
\begin{split}
(q,t) &\in B \left( p_{i}, t_{i}, \frac{\epsilon_{2}}{\sqrt{2 C^{-1}}} r_{\mathrm{Ric}}(p_{i}, t_{i}) \right) \\
&\times \left( \max \left\{ 0, t_{i} - \frac{\epsilon_{2} (\sqrt{2 C^{-1}})^{-1}}{1 + C_{1}} r^{2}_{\mathrm{Ric}}(p_{i}, t_{i}) \right\},
\min \left\{ T,  t_{i} + \frac{\epsilon_{2} (\sqrt{2 C^{-1}})^{-1}}{1 + C_{1}} r^{2}_{\mathrm{Ric}}(p_{i}, t_{i}) \right\} \right)
\end{split}
\]
such that
\[
|\mathrm{Ric}| (q, t_{i}) \ge \epsilon_{1} r_{\mathrm{Ric}}^{-2}(p_{i},t_{i}),
\]
where $C_{1} = 2 ^{4}\sqrt{2/3}\sqrt{n-1}.$
\end{theo}

\begin{proof}[Proof of Main Theorem from Theorem \ref{theo1}]
As seen in Section 3, if there exists a point $p \in \Sigma_{I},$ then $p$
satisfies $(\mathrm{Rm/Ric}).$
Hence, Main Theorem follows from Theorem \ref{theo1}.
\end{proof}

\begin{proof}[Proof of Theorem \ref{theo1}]
Firstly, we will show the case of $\epsilon_{2} = 1/2.$
We will prove this by contradiction.
Suppose that the claim is false.
Then, there exists $\epsilon_{1} \in \left( 0, \frac{1}{\sqrt{n}(n-1)} \right),$
a blow-up sequence $(p_{i}, t_{i}) \in M \times [0,T)$ such that $p_{i} \rightarrow p$ and $t_{i} \rightarrow T$
and $i_{0} \in \mathbb{N}$ such that for all $i \ge i_{0}$ there exists 
\[
\begin{split}
(q,t) &\in B \left( p_{i}, t_{i}, \frac{1/2}{\sqrt{2 C^{-1}}} r_{\mathrm{Ric}}(p_{i}, t_{i}) \right) \\
&\times \left( \max \left\{ 0, t_{i} - \frac{1/2 (\sqrt{2 C^{-1}})^{-1}}{1 + C_{1}} r^{2}_{\mathrm{Ric}}(p_{i}, t_{i}) \right\},
\min \left\{ T,  t_{i} + \frac{1/2 (\sqrt{2 C^{-1}})^{-1}}{1 + C_{1}} r^{2}_{\mathrm{Ric}}(p_{i}, t_{i}) \right\} \right)
\end{split}
\]
we have 
\[
|\mathrm{Ric}|(q, t_{i}) \ge \epsilon_{1} r^{-2}_{\mathrm{Ric}}(p_{i}, t_{i}).
\]
Thus, from Theorem \ref{buzano-matteo} and Remark \ref{rema-well},
by taking $i \in \mathbb{N}$ sufficiently greater than $i_{0}$ if necessary,
we can take a point $(x, t_{i}) \in B(p_{i}, t_{i}, \delta r_{\mathrm{Ric}}(p_{i}, t_{i}))$
where $\delta := \min \left\{ \frac{\epsilon_{1}}{6}, \frac{1}{2(1+C_{1})} \right\}.$
such that the following holds :

\noindent
$(0)~\frac{r^{-2}_{\mathrm{Ric}}(p_{i},t_{i})}{r^{-2}_{\mathrm{Rm}}(p_{i},t_{i})} \ge \frac{1}{2}C.$

\noindent
This is deduced from $(\mathrm{Rm/Ric}),$ Lemma \ref{prop1}, the distance distortion estimate Lemma \ref{prop3} and Lemma \ref{prop8}.

\[
(1)~r_{\mathrm{Ric}}(p_{i},t_{i}) 
\le \frac{C}{32} (A_{0} + A_{1})^{-1} (1 + \bar{C})^{-1} \delta.
\]

\noindent
where $C_{1}, \bar{C}, A_{i}$ denote respectively the universal constant in Lemma \ref{prop1}, Lemma \ref{prop3} and Lemma \ref{prop7}.

\noindent
This is deduced from the distance distortion Lemma \ref{prop3}, Lemma \ref{prop8} and $(\mathrm{Rm/Ric}).$

Since 
\[
\begin{split}
(q, t) &\in B \left( p_{i}, t_{i}, \frac{1/2}{\sqrt{2 C^{-1}}} r_{\mathrm{Ric}}(p_{i}, t_{i}) \right) \\
&\times \left( \max \left\{ 0, t_{i} - \frac{1/2 (\sqrt{2 C^{-1}})^{-1}}{1 + C_{1}} r^{2}_{\mathrm{Ric}}(p_{i}, t_{i}) \right\},
\min \left\{ T,  t_{i} + \frac{1/2 (\sqrt{2 C^{-1}})^{-1}}{1 + C_{1}} r^{2}_{\mathrm{Ric}}(p_{i}, t_{i}) \right\} \right),
\end{split}
\]
by (0), we have
\[
\begin{split}
(q,t) &\in B \left( p_{i}, t_{i}, \frac{1/2}{\sqrt{2 C^{-1}}} r_{\mathrm{Rm}}(p_{i}, t_{i}) \right) \\
&\times \left( \max \left\{ 0, t_{i} - \frac{1/2 (\sqrt{2 C^{-1}})^{-1}}{1 + C_{1}} r^{2}_{\mathrm{Rm}}(p_{i}, t_{i}) \right\},
\min \left\{ T,  t_{i} + \frac{1/2 (\sqrt{2 C^{-1}})^{-1}}{1 + C_{1}} r^{2}_{\mathrm{Rm}}(p_{i}, t_{i}) \right\} \right).
\end{split}
\]
Hence, by lemma \ref{prop2} and Remark \ref{rema-harnack}, we have
\[
r^{-2}_{\mathrm{Rm}}(q,t) \le 16 r^{-2}_{\mathrm{Rm}}(p_{i}, t_{i}).
\]
Thus, by (0),
\[
r^{-2}_{\mathrm{Rm}}(q,t) \le \frac{32}{C} r^{-2}_{\mathrm{Ric}}(p_{i}, t_{i}).
\]
Hence, by Lemma \ref{prop7} (and Remark \ref{rema1}), this estimate and $(1)$, we have 
\[
|\mathrm{Ric}|(x, t_{i}) \ge |\mathrm{Ric}|(q, t) - \left( A_{0} + A_{1} \right) (1 + \bar{C}) r_{\mathrm{Ric}}(p_{i}, t_{i}) \cdot r^{-2}_{\mathrm{Ric}}(p_{i}, t_{i}).
\]
Here, from the choice of $(x, t_{i})$ and lemma \ref{prop2}, we have
\[
|\mathrm{Ric}|(x, t_{i}) \le 4 \delta r^{-2}_{\mathrm{Ric}}(p_{i}, t_{i}).
\]
Therefore, by the above inequality, the defition of $\delta$ and (1), we have
\[
\epsilon_{1} r_{\mathrm{Ric}}^{-2}(p_{i}, t_{i}) > 5 \delta r_{\mathrm{Ric}}^{-2}(p_{i}, t_{i}) \ge |\mathrm{Ric}|(q,t).
\]
On the other hand, since
\[
|\mathrm{Ric}|(q,t) \ge \epsilon_{1} r_{\mathrm{Ric}}^{-2}(p_{i}, t_{i}).
\]
This contradist the above inequality.
Therefore we obtain the assertion for the case of $\epsilon_{2} = 1/2.$
And, from this case, we also get the assertion for the case of $\epsilon_{2} \le 1/2.$

For the case of $\epsilon_{2} \in (1/2, 1),$ we can show it in the almost same way as above.
We will explain it as follows.
If we suppose that the claim is false.
Then, there exists $\epsilon_{1} \in \left( 0, \frac{1}{\sqrt{n}(n-1)} \right),$
a blow-up sequence $(p_{i}, t_{i}) \in M \times [0,T)$ such that $p_{i} \rightarrow p$ and $t_{i} \rightarrow T$
and $i_{0} \in \mathbb{N}$ such that for all $i \ge i_{0}$ there exists 
\[
\begin{split}
(q,t) &\in B \left( p_{i}, t_{i}, \frac{\epsilon_{2}}{\sqrt{2 C^{-1}}} r_{\mathrm{Ric}}(p_{i}, t_{i}) \right) \\
&\times \left( \max \left\{ 0, t_{i} - \frac{\epsilon_{2} (\sqrt{2 C^{-1}})^{-1}}{1 + C_{1}} r^{2}_{\mathrm{Ric}}(p_{i}, t_{i}) \right\},
\min \left\{ T,  t_{i} + \frac{\epsilon_{2} (\sqrt{2 C^{-1}})^{-1}}{1 + C_{1}} r^{2}_{\mathrm{Ric}}(p_{i}, t_{i}) \right\} \right),
\end{split}
\]
we have 
\[
|\mathrm{Ric}|(q, t_{i}) \ge \epsilon_{1} r^{-2}_{\mathrm{Ric}}(p_{i}, t_{i}).
\]
And we can take a point $(x, t_{i})$ ($i$ is sufficiently large) in the same way as above
which satisfies (0) and (1).
Thus, using Lemma \ref{prop2}, Remark \ref{rema-harnack} $m$ times and arguing as above, we obtain 
\[
r^{-2}_{\mathrm{Rm}}(x, t_{i}) \le C (C, m) r^{-2}_{\mathrm{Ric}}(p_{i}, t_{i})
\]
for some positive large positive constant $C = C(C, m)$ depend on $C$ and $m.$
If we use Lemma \ref{prop2} and Remark \ref{rema-harnack} $m$ times,
then the above estimate holds for all
\[
\begin{split}
(x, t_{i}) &\in B \left( p_{i}, t_{i}, \frac{\epsilon_{m}}{\sqrt{2 C^{-1}}} r_{\mathrm{Ric}}(p_{i}, t_{i}) \right) \\
&\times \left( \max \left\{ 0, t_{i} - \frac{\epsilon_{m} (\sqrt{2 C^{-1}})^{-1}}{1 + C_{1}} r^{2}_{\mathrm{Ric}}(p_{i}, t_{i}) \right\},
\min \left\{ T,  t_{i} + \frac{\epsilon_{m} (\sqrt{2 C^{-1}})^{-1}}{1 + C_{1}} r^{2}_{\mathrm{Ric}}(p_{i}, t_{i}) \right\} \right),
\end{split}
\]
where $\epsilon_{m} := 1 - \left( \frac{1}{2} \right)^{m}.$
So if we take $m$ sufficiently large so that $\epsilon_{m} > \epsilon_{2}$
and teke $i$ large enough so that 
\[
(1')~r_{\mathrm{Ric}}(p_{i},t_{i}) 
\le C(C, m)^{-1} (A_{0} + A_{1})^{-1} (1 + \bar{C})^{-1} \delta
\]
instead of (1).
Then we obtain a contradiction in the same way as above.
This completes the proof.
\end{proof}

\begin{rema}
In the Main Theorem, for any blow-up sequence $(p_{i}, t_{i}),$
by Lemma \ref{prop1} and Lemma \ref{sing3},
\[
\begin{split}
r_{\mathrm{Ric}}(p_{i}, t_{i}) &\le r_{\mathrm{Ric}}(p, t_{i}) + (1 + \bar{C}) r_{\mathrm{Ric}}(p, t_{i}) \\
&= (2 + \bar{C}) r_{\mathrm{Ric}}(p, t_{i}) \\
&\le (2 + \bar{C}) \sqrt{T - t_{i}},
\end{split}
\]
where $\bar{C}$ denotes the constant appeared in distance distortion estimates (Lemma \ref{prop3}).
Hence, if we take $C$ large enough (we can take this arbitrary large if necessary from the definition of $C$)
so that $\frac{\epsilon_{2} (2 + \bar{C})^{2} }{\sqrt{2 C}(1 + C_{1})} < 1/2$ and
$i$ large enough so that $\frac{T}{2} < t_{i},$ then
\[
\begin{split}
t_{i} + \frac{\epsilon_{2}}{\sqrt{2 C}(1 + C_{1})} r_{\mathrm{Ric}}^{2}(p_{i}, t_{i}) &\le t_{i} + \frac{\epsilon_{2} (2 + \bar{C})^{2} }{\sqrt{2 C}(1 + C_{1})} (T - t_{i}) \\
< t_{i} + \frac{1}{2} (T - t_{i}) < T
\end{split}
\]
and
\[
\begin{split}
t_{i} - \frac{\epsilon_{2}}{\sqrt{2 C}(1 + C_{1})} r_{\mathrm{Ric}}^{2}(p_{i}, t_{i}) &\ge  t_{i} - \frac{\epsilon_{2} (2 + \bar{C})^{2} }{\sqrt{2 C}(1 + C_{1})} (T - t_{i}) \\
> t_{i} - \frac{1}{2} (T - t_{i}) > 0.
\end{split}
\]
That is, in this setting, 
\[
\max \left\{ 0, t_{i} - \frac{\epsilon_{2} (\sqrt{2 C})^{-1}}{1 + C_{1}} r^{2}_{\mathrm{Ric}}(p_{i}, t_{i}) \right\} 
= t_{i} - \frac{\epsilon_{2} (\sqrt{2 C})^{-1}}{1 + C_{1}} r^{2}_{\mathrm{Ric}}(p_{i}, t_{i})
\]
and
\[
\min \left\{ T, t_{i} + \frac{\epsilon_{2} (\sqrt{2 C})^{-1}}{1 + C_{1}} r^{2}_{\mathrm{Ric}}(p_{i}, t_{i}) \right\}
= t_{i} + \frac{\epsilon_{2} (\sqrt{2 C})^{-1}}{1 + C_{1}} r^{2}_{\mathrm{Ric}}(p_{i}, t_{i}).
\]
\end{rema}

\subsection*{Acknowledgement}
~~I would like to thank my supervisor Kazuo Akutagawa for continuous support.
I also would like to thank Reto Buzano for pointing out my mistake in the first version of this paper.

\bigskip
\textit{E-mail adress}:~hamanaka1311558@gmail.com

\textsc{Department Of Mathematics, Chuo University, Tokyo 112-8551, Japan}

\end{document}